\documentclass[a4paper, 11pt]{amsart}
\usepackage{hyperref,xy,doi}
\usepackage{ytableau, varwidth}
\usepackage{diagbox}
\usepackage{soul}
\usepackage{enumitem}
\xyoption{all}
\usepackage{amsmath,amssymb, xy,wasysym,bm, amsthm}
\usepackage{tikz}
\usetikzlibrary{cd}
\usepackage{color}
\usepackage{multirow}
\usepackage[left=3cm, right=3cm]{geometry}
\theoremstyle{plain}
\newtheorem{theorem}{Theorem}[section]
\theoremstyle{definition}
\newtheorem{corollary}[theorem]{Corollary}
\newtheorem{lemma}[theorem]{Lemma}
\newtheorem{definition}[theorem]{Definition}
\newtheorem{remark}[theorem]{Remark}
\newtheorem{prop}[theorem]{Proposition}

\newtheorem{cor}[theorem]{Corollary}
\newtheorem{lem}[theorem]{Lemma}

\newtheorem{que}[theorem]{Question}
\newtheorem{example}[theorem]{Example}

\newcommand{\wchi}{\widetilde{\chi}}
\newcommand{\mn}[1]{({\bf MN#1})}

\title[Quasi Steinberg character]{On Quasi Steinberg characters of Symmetric and Alternating groups and their Double Covers}
\author[Paul]{Digjoy Paul}
\address{Tata Institute of Fundamental Research, Mumbai, India.}
\email{digjoypaul@gmail.com}
\author[Singla]{Pooja Singla}
\address{ Department of Mathematics and Statistics, Indian Institute of Technology Kanpur, Kanpur 208016, India. }
\email{psingla@iitk.ac.in}

\begin{document}

\begin{abstract}
An irreducible character of a finite group $G$ is called quasi $p$-Steinberg character for a prime $p$  if it takes a nonzero value on every $p$-regular element of $G$. In this article, we classify the quasi $p$-Steinberg characters of Symmetric ($S_n$) and Alternating ($A_n$) groups and their double covers. In particular, an existence of a non-linear quasi $p$-Steinberg character of $S_n$ implies $n \leq 8$ and of $A_n$ implies $n \leq 9$. 
\end{abstract}	
\maketitle
\section{Introduction }
 Let $G$ be a finite group and $p$ be a prime dividing $|G|$. A $p$-Steinberg character of $G$ is an irreducible character $\theta$ of $G$ such that $\theta(x) = \pm |C_G(x)|_p,$
  for every  $p$-regular element $x$ in $G$, i.e, $x$  having order co-prime to $p$.  Here $C_G(x)$ denotes the centralizer of $x$ in $G$ and $|n|_p$ denotes the $p$-part of an integer $n$. Since $\theta(1)=|G|_{p}$, it has defect $0$. Therefore $ \theta (y) = 0$ for all $y \in G$ such that $p\mid\mathrm{ord}(y)$, see Brauer--Nesbitt~ \cite[Theorem~1]{MR4042}.  The above character coincides with the well known Steinberg character for the finite groups of Lie type. We refer the reader to Steinberg~\cite{MR80669, MR87659}, Curtis~\cite{MR201524}, and beautiful exposition of Humphreys~\cite{MR876960} for more properties regarding Steinberg characters of finite groups of Lie type. See also Carter \cite{MR794307} concerning similar results for finite groups with $(B, N)$ pair. Feit~\cite{MR1247494} proved that the set of rational valued linear characters and $p$-Steinberg characters of  $G$ are in bijection if $G$ has a $p$-Steinberg character for $p \neq 2$.  While speculating that $p$-Steinberg characters play role in determining the structure of a group, Feit conjectured that if a finite simple group $G$ has a $p$-Steinberg character then $G$ is isomorphic to a simple group of Lie type in characteristic $p$. This question was positively answered  by Darafasheh~\cite{MR1382032} for $A_n$ and $\mathrm{PSL}_n(q)$ and by Tiep  ~\cite{MR1425574} for remaining finite simple groups.

 It is natural to ask about the existence of characters that satisfy a weaker hypothesis than the $p$-Steinberg characters. An irreducible character $\theta$ of a finite group $G$ is called $p$-vanishing if  $\theta(x) = 0$ whenever $p \mid \mathrm{ord}(x)$. A $p$-vanishing character $\theta$ is called Steinberg-like   
 $\theta(1) = |G|_p$. Pellegrini and Zalleski \cite{MR3513936} proved that for any simple group of Lie type in characteristic $p$ except $B_n(q)$ for $n=3,4,5$ and $D_n(q)$ for $n=4,5$, $\theta$ is a Steinberg-like character iff it is $p$-Steinberg . Malle and Zalleski \cite{MR4048469} classified Steinberg-like characters of all finite simple group except for $A_n$ with $p=2$. In this article, we consider the following generalizations of $p$-Steinberg characters of finite groups.
\begin{definition} (quasi $p$-Steinberg character) Let $G$ be a finite group and $p$ be a prime dividing $|G|$. An irreducible character $\chi$ of $G$  is called {\it quasi p-Steinberg character} if $\chi(g) \neq 0$ for all $g \in G$ such that $p \nmid \mathrm{ord}(g)$.  
\end{definition}
Motivated by a question of Dipendra Prasad we also consider {\it weak $p$-Steinberg characters} of $G$.
\begin{definition}(weak $p$-Steinberg character)
For a finite group $G$, an irreducible character $\theta$ of $G$ is called weak $p$-Steinberg character if $\theta$ is  Steinberg-like and quasi $p$-Steinberg character.
\end{definition}
See Remark \ref{rmk:variant stg} for relation among all these variant of Steinberg characters. For any finite group $G$ and prime $p$ dividing $|G|$, any linear (one dimensional) character of $G$ is a quasi $p$-Steinberg character. Therefore, we will mainly focus on the non-linear quasi $p$-Steinberg characters. Let $\chi_\lambda$ be the irreducible character of the Specht module $V_\lambda$ of $S_n$ corresponding to a partition $\lambda$ of $n$. Our first main result characterizes the $\ lambda$s such that $\chi_\lambda $ is a quasi $p$-Steinberg character of $S_n$.

\begin{theorem}
	\label{thm:classification-symmetric}
For $n \geq 3$, let $\lambda$ be a partition of $n$ such that $\lambda \neq (n), (1^n)$ and  $p$ be a prime. All tuples $(n, \lambda , p)$ such that $\chi_\lambda$ is a quasi $p$-Steinberg character of $S_n$ is given in Table~\ref{tab:quasi-Sn}. 

\begin{table}[ht] 
\begin{tabular}[t]{c|c|c}
	 \,\,\,\,\, \,\,\, $n$ \,\,\,\,\, \,\,\,& \,\,\,\,\, \,\,$\lambda$  \,\,\,\,\, \,\, &\,\,\,\,\, \,\,\, $p$  \,\,\,\,\, \,\,\, \\
	\hline \hline 
	$3$ & $(2,1)$ & $2$ \\
	\hline 
	$4$ & $(2,2)$ & $2$ \\
	$4$ & $(3,1)$, $(2,1,1)$  & $3$ \\
	\hline 
	$5$ &  $(4,1)$, $(2,1,1,1)$ & $2$ \\
	$5$ & $(3,2)$, $(2,2,1)$ & $5$ \\
	\hline 
	$6$ & $(3,2,1)$ & $2$ \\
	$6$ & $(4,2)$, $(2,2,1,1)$ & $3$ \\
	\hline 
	$8$ & $(5,2,1), (3,2,1,1,1)$ & $2$ \\
\hline 
	\end{tabular}
\vspace{.2cm} 
\caption{quasi $p$-Steinberg characters of $S_n$}\label{tab:quasi-Sn}
\end{table}
\end{theorem}
Our next result characterizes quasi $p$-Steinberg characters of $A_n$.
For the self conjugate $\lambda$, let $\chi_\lambda^\pm$ be distinct irreducible characters of $A_n$ such that $\chi_\lambda|_{A_n} = \chi_\lambda^+  + \chi_\lambda^-$. For non self-conjugate $\lambda$, let $\chi_\lambda^\downarrow $ be an irreducible character of $A_n$ given by $\chi_\lambda^\downarrow = \chi_\lambda|_{A_n}$. 

\begin{theorem}
	\label{thm:classification-alternating}
For $n \geq 3$, let $\lambda $ be a partition of $n$ such that $\lambda \neq (n), (1^n)$.  Let $* \in \{ \pm, \downarrow \}$ and $p$ be a prime. All tuples $(n, \lambda, p)$ such that $\chi_\lambda^*$ is quasi $p$-Steinberg character of $A_n$ are included in Table~\ref{tab:quasi-An}. 

\begin{table}[ht] 
\begin{tabular}[t]{c|c|c}
\,\,\,\,\, \,\,\, $n$ \,\,\,\,\, \,\,\,& \,\,\,\,\, \,\,$\lambda$  \,\,\,\,\, \,\, &\,\,\,\,\, \,\,\, $p$  \,\,\,\,\, \,\,\, \\
\hline \hline 
$3$ & $(2,1)$ & $3$ \\ 
\hline 
$4$ & $(2,2)$ & $2$ \\
$4$ & $(3,1), (2,2)$ & $3$ \\
\hline 
$5$ & $(4,1)$& $2 $\\
$5 $& $(3,1,1)$ & $3$ \\
$5 $& $(3,2)$ & $5$ \\
\hline 
$6$ & $(3,2,1)$ & $2$ \\
$6$ & $(4,2)$ & $3 $\\
$6$ & $(5,1), (3,3)$ & $5$ \\
\hline 
$8$ & $(5,2,1)$ & $2 $\\
\hline 
$9$ & $(7,2)$& $3$ \\
\hline 
\end{tabular}
\vspace{.2cm} 
\caption{quasi $p$-Steinberg characters of $A_n$}\label{tab:quasi-An}
\end{table}

\end{theorem}
The next result characterizes the quasi $p$-Steinberg spin characters of double covers $\widetilde{S}_n$ and $\widetilde{A}_n$ of $S_n$ and $A_n$ respectively. For the notations regarding $\widetilde{S}_n$, $\widetilde{A}_n$ and spin characters, we refer the reader to Section~\ref{sec:preliminaries}. 
\begin{theorem}
	\label{thm:classification-double-cover}
	For $n \geq 4$, let $\lambda$ be a partition of $n$ into distinct parts and $p$ be a prime and $* \in \{\pm, \downarrow  \}$. An irreducible spin character $\widetilde{\chi}_\lambda$  $(\text{respectively}, \langle\wchi_\lambda\rangle ^* )$ of $\widetilde{S}_n$ $(\text{respectively}, \widetilde{A}_n)$ is quasi $p$-Steinberg if and only if $p = 2$ and either $\lambda = (n)$ or is one of the following:
	\[
	(3,1), (3,2), (3,2,1), (5,1), (5,2,1). 
	\] 
\end{theorem}
The proof of these results rely on the combinatorial representation theory of $S_n$, $A_n$ and their double covers. In particular, Murnaghan--Nakayama rule and Morris recursion formula play a prominent role in these proofs.
\begin{remark}
For $p=2$, the classification of quasi $p$-Steinberg characters of $S_n$ and its double cover can be deduced from Bessenrodt--Olsson~\cite[Theorem 1.1]{MR2120100} and Bessenrodt~\cite[Thorem 1.1]{MR2335704} respectively. See Section~\ref{section:further-questions} for more details.
\end{remark}
The following is an obvious corollary of Theorem~\ref{thm:classification-symmetric}.  Analogous result for $A_n$, $\widetilde{S}_n$ and $\widetilde{A}_n$ can also be obtained from Theorems~\ref{thm:classification-alternating}, \ref{thm:classification-double-cover}. 
\begin{corollary}
	For $n \geq 9$ and $p \leq n$, every non-linear character $\chi_\lambda$ of $S_n$ has a zero at some $p$-regular element of $S_n$, i.e. there exists $p$-regular $g\in S_n$ such that $\chi_\lambda(g) = 0$. 
\end{corollary}

The next two results characterize the weak $p$-Steinberg characters of $S_n$, $A_n$ and their double covers. 
\begin{corollary}
	\label{thm:weak-Steinberg-symmetric-alternating} For $n \geq 3$, let $\lambda $ be a partition of $n$ and $p$ be a prime.
	\begin{enumerate} 
		\item Every pair $[\lambda, p]$ such that $\chi_\lambda$ is a quasi $p$-Steinberg character  of $S_n$ is also a weak $p$-Steinberg character of $S_n$ except the following:
		\[
		[(2,2),2], [(4,1), 2], [(2,1,1,1), 2], [(5,2,1), 2], [(3,2,1^3), 2].
		\]
		\item Every pair $[\lambda, p]$ such that $\chi_\lambda^*$ is a quasi $p$-Steinberg character  of $A_n$ is also a weak $p$-Steinberg character of $A_n$ except the following:
		\[
		[(2,1), 3], [(2,2), 2], [(2,2),3], [(2,2,1^5),3]. 
		\]

	\end{enumerate} 

\end{corollary}
	\begin{corollary}
			\label{thm:weak-Steinberg-double-cover}
	For $n \geq 4$, there does not exist any weak $p$-Steinberg spin character of $\widetilde{S}_n$ and $\widetilde{A}_n$. 
\end{corollary}

\begin{remark}
	\label{remark:weak-alternate-proof}
	We remark that an alternate way to prove Corollaries~\ref{thm:weak-Steinberg-symmetric-alternating} and \ref{thm:weak-Steinberg-double-cover} is by using the main results of Balog--Bessenrodt--Olsson--Ono~\cite{MR1853455} and Bessenrodt--Olsson~\cite{MR1920404}.   We leave these details for the reader. 
\end{remark}
\begin{remark}
\label{rmk:variant stg}
 Every $p$-Steinberg character is  a weak $p$-Steinberg and therefore also is a quasi $p$-Steinberg character. But the converse need not hold. Further, the notion of  Steinberg-like character is also different from that of quasi $p$-Steinberg character. Below we give examples of these phenomena. 
	\begin{enumerate} 
	 \item The representation $V_\lambda$ for $\lambda = (2,2)$ of $S_4$ is a quasi $2$-Steinberg representation of $S_4$ but is not a weak $2$-Steinberg representation. 
	 \item The representation $V_{(3,1, 1)}^+$ is a weak $3$-Steinberg representation of $A_5$ but is not a $3$-Steinberg representation of  $A_5$. 
	 \item The representation $V_\lambda$ for $\lambda = (5,1)$ of $S_6$ is Steinberg-like for $p = 5$, however it is not a quasi $p$-Steinberg representation for any $p \mid 6!$ (see Theorem~\ref{thm:classification-symmetric}). 
	 
	 \end{enumerate} 
\end{remark}
This article is organized as follows. In Section~\ref{sec:preliminaries}, we collect a few preliminary results regarding the representations of $S_n$, $A_n$ and their double covers. The proofs of all main results are included in Section~\ref{section:proofs}. At last, in Section~\ref{section:further-questions} we collect a few questions arising naturally from this work.

\section{Preliminaries}
\label{sec:preliminaries}
In this section, we collect the notations and basic results regarding character values  for $S_n$, $A_n$, and their double covers (see Olsson~ \cite{MR1264418}, Prasad~\cite{MR3287258}, Stanley~\cite{MR1676282},  for more details). In particular,  we recall Murnaghan--Nakayama rule for $S_n$ (Theorem~\ref{thm:MNrule}) and Morris recursion formula for $\widetilde{S}_n$ (Theorem~\ref{thm:M-N-double}). 
\subsection{Symmetric group representation}
Any element (permutation) in $S_n$ can be written as a product of disjoint cycles. The partition determined by such a product is called the cycle type of the corresponding permutation. Two permutations with the same cycle type are in the same conjugacy class. There is a bijective correspondence between the conjugacy classes of $S_n$ and the set  $\Lambda(n)$ of all partitions of $n$. Given a partition $\alpha=(\alpha_1, \ldots ,\alpha_l)$ of $n$, let

\[
w_\alpha =  (1,2, \ldots, \alpha_1)(\alpha_1 +1,\ldots, \alpha_1+\alpha_2) \cdots (\alpha_1  + \cdots + \alpha_{l-1}+1, \ldots, \alpha_1   + \cdots  + \alpha_l )
\]
denote a representative of the conjugacy class $C_{\alpha}$ indexed by $\alpha$. We shall write $type(w)$, the type of an element $w$, to mean the corresponding cycle type of $w$.

  The set $\Lambda(n)$ also indexes the set of all irreducible representations of $S_n$. For $\lambda \in \Lambda(n)$, let $V_{\lambda}$ (respectively, $\chi_{\lambda}$) denote the corresponding irreducible representation (respectively, irreducible character) of $S_n$. The dimension of $V_{\lambda}$ is given by the Hook length formula
\[\chi_{\lambda}(1)=\frac{n!}{H_{\lambda}} \,  \text{ where } H_{\lambda}=\prod_{(i,j)\in \lambda}h_{ij}, 
\]
where the product runs over all cells $(i,j)$ of $Y(\lambda)$, the Young diagram of $\lambda$ and $h_{ij}$ is the hook length of the cell $(i,j)$, the number
of boxes in $Y(\lambda)$ that lie directly below or directly to its right plus one. 

Let $Rim_{ij}$, known as \emph{rimhook} at $(i,j)$, denote the set of boxes of $Y(\lambda)$ in positions $(k,l)$ such that $k\geq i, l\geq j$ and the Young diagram does not have a box in position $(k + 1, l + 1)$. Note that $h_{ij}=|Rim_{ij}|$ and once we remove a rimhook $Rim_{ij}$ from $Y(\lambda)$, the resultant diagram $Y(\lambda) \setminus Rim_{ij}$ still gives a valid partition of size $n-h_{ij}$. For a given integer $m \in \mathbb N$,  the $m$-weight of $\lambda$, denoted by $wt_m(\lambda)$, is the maximal number of rimhooks of size $m$ that can successively be removed from $\lambda$ and is equal to the number of hooks in $\lambda $ of length $m$.

For example,
for $\lambda=(3,3,2)$, $h_{12}=4$ and the corresponding rimhook $Rim_{12}$ consist of 4 boxes marked with `X' in the first diagram. In the second   
diagram, named \emph{hook diagram}, we put corresponding hook numbers of the cells. Here $H_{\lambda}=5\cdot4\cdot 2\cdot 4 \cdot 3 \cdot 1 \cdot 2 \cdot 1$. Moreover, $wt_2(\lambda)=2$ as there are two hooks of length 2 in $H(\lambda)$.
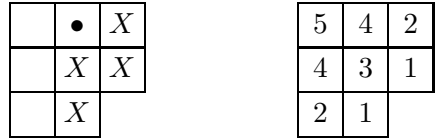
\begin{displaymath}
\ytableausetup{centertableaux}
\ytableaushort
{\none \bullet X,\none X X,\none X}
* {3,3,2} 
\hspace{2cm}
\ytableausetup{centertableaux}
\ytableaushort
{542,431,21}
\end{displaymath}
We shall use $\chi_{\lambda}(\alpha)$ to denote the value of the character $\chi_{\lambda}$ at the conjugacy class of $w_{\alpha}$.
\begin{theorem}[Murnaghan--Nakayama rule]
	\label{thm:MNrule}
Given a permutation of cycle type  $\alpha=(\alpha_1,\ldots,\alpha_l)$ and a partition $\lambda$ of $n$, we have
\begin{eqnarray*}
\chi_{\lambda}(\alpha) = \sum_{\nu} (-1)^{ht(\nu)-1} \chi_{\lambda \setminus \nu}( \alpha \setminus \alpha_i ),
\end{eqnarray*}
where the sum runs over all rimhooks $\nu$ of length $\alpha_i$ and $ht(\nu)$ is the height of $\nu$ (equals to the number of rows in the skew-diagram of $\nu$), and $\alpha \setminus \alpha_i=(\alpha_1,\ldots,\alpha_{i-1},\alpha_{i+1},\ldots, \alpha_l)$.
\end{theorem}

\begin{example}
	Let $\lambda=(3,3)$ and $\alpha=(4,2)$. The hook diagram and possible 2-rimhooks are given below.
	\begin{displaymath}
	\ytableausetup{centertableaux}
	\ytableaushort
	{432,321}
	\hspace{2cm}
	\ytableausetup{centertableaux}
	\ytableaushort
	{\none \none X,\none \none X}
	* {3,3} 
	\hspace{2cm}
	\ytableaushort
	{\none \none \none, \none X  X}
	* {3,3} 
	\end{displaymath}
	
	Then Murnaghan--Nakayama rule can be used to determine $\chi_{(3,3)}((4,2))$ as follows.  $$\chi_{(3,3)}((4,2))=(-1)^{(2-1)} \chi_{(2,2)}((4)) + (-1)^{(1-1)} \chi_{(3,1)}((4)),$$
	as the first rim-hook has height $2$ and the other one has height $1$. From the hook diagram of $(2,2)$ below, we get $\chi_{(2,2)}((4))=0$ as $wt_4((2,2)) = 0$.
	\begin{displaymath}
	\ytableausetup{centertableaux}
	\ytableaushort
	{32,21}
	\hspace{2cm}
	\ytableausetup{centertableaux}
	\ytableaushort
	{421,1}
	\hspace{2cm}
	\ytableausetup{centertableaux}
	\ytableaushort
	{X X X,X}
	* {3,1} 
	\end{displaymath}
	Since $wt_4((3,1)) = 1$, we get $\chi_{(3,1)}((4))=\chi_{()}(\empty)=1$. Hence $\chi_{(3,3)}((4,2))=1$.
\end{example}
The next corollary is an immediate consequences of the Murnaghan--Nakayama rule.
\begin{corollary}
\label{MN12}
\noindent {\bf (MN1):} For  two partitions $\lambda$ and $\alpha=(\alpha_1,\ldots,\alpha_l)$ of $n$, $\chi_{\lambda}(\alpha)=0$ if $wt_{\alpha_i}(\lambda) = 0$ for some $i \in \{1, \ldots, l\}$.

\noindent {\bf (MN2):} For $\alpha=(\alpha_1,\alpha_2) \vdash n$, we have
	$\chi_{\lambda}(\alpha)=0$ if $wt_{\alpha_1}(\lambda)  = 1$ and $wt_{\alpha_2} (\lambda \setminus \nu) = 0$, where $\nu$ is the unique rim hook of $\lambda$ of length $\alpha_1$.
\end{corollary}

\subsection{Alternating group representations}
The Alternating group $A_n$ is the subgroup of $S_n$ consisting of even permutations. A partition $\alpha \vdash n$ is called  even (respectively, odd) if $w_\alpha \in A_n$ (respectively, $w_\alpha \in S_n\setminus A_n$). The next result gives information regarding the conjugacy classes of $A_n$. 

\begin{theorem}{\cite[Theorem~4.6.13]{MR3287258}} For each even partition $\lambda$ of $n$ with distinct odd parts, the set of permutations
with cycle type $\lambda$ splits into two conjugacy classes in $A_n$, $C_\lambda^+$ and $C_\lambda^-$
of equal cardinality. In fact, for any odd permutation $x \in S_n$,
\[
C_\lambda^- = x C_\lambda^+ x^{-1}.
\]
For any other even partition $\lambda$, the set of permutations with cycle type $\lambda$ is a conjugacy class in $A_n$. 
\end{theorem} 

The irreducible representations of $A_n$ are also characterized by $\lambda \vdash n$. For $\lambda = \lambda'$, the conjugate of $\lambda,$ there are two irreducible representations $V_{\lambda}^+$  and $V_{\lambda}^-$ of $A_n$ such that ${V_\lambda}|_{A_n}=V_{\lambda}^+ \oplus V_{\lambda}^-$. Let $\chi_\lambda^{\pm}$ be the corresponding irreducible characters. For each $\lambda \neq \lambda'$, the restriction $V_\lambda |_{A_n}$ is an irreducible representation of $A_n$. Let $\chi_\lambda^{\downarrow}$ be the corresponding irreducible character in this case, see \cite[Theorem~4.6.7]{MR3287258} for more details. In general, we shall use $\chi_{\lambda}^*$ with $* \in \{\pm, \downarrow   \}$ to denote these irreducible characters of $A_n$.

\begin{definition} (Folding)
	Given a partition with all distinct parts of odd lengths, we can fold (uniquely) each part to make a self-conjugate hook and a self-conjugate partition. Denote the folding operation by $\mathcal{F}$. 
\end{definition}	
	
	For example, folding the first part of length $7$ of  $\lambda=(7,3)$ we get the self-conjugate hook $(4,1,1,1)$ and by folding the other part $3$ we get $(2,1)$. Hence we obtain a self-conjugate partition $\mu=(4,3,2,1)$. 
	\begin{displaymath}
	\ytableausetup{centertableaux}
	\ytableaushort
	{1111111,222}	
	\quad
    \overset{\mathcal{F}}{\longmapsto}
    \quad
	\ytableausetup{centertableaux}
	\ytableaushort
	{1111,122,12,1}.
	\end{displaymath}
	We observe that the partition obtained by folding is unique, and the process is reversible. Hence we obtain the following result:
	
\begin{lemma}
	For $\lambda, \mu \vdash n$ such that $\lambda = \lambda'$. Then $\mathcal{F}(\mu ) = \lambda$ implies $\mu$ has distinct odd parts. 
\end{lemma}

The following result relates the character values of $A_n$ with that of $S_n$ and is used repeatedly in Section~\ref{section:proofs}. 

\begin{lem} 
	\label{lemma:deduction-ofsymmetric-nonzero}
	
	For each $w \in A_n$, we have the following:
	\[ 
	\chi_\lambda^{*}(w) \neq 0\,\,\, \mathrm{if \,\,and \,\, only\,\, if} \,\,\,  \chi_\lambda(w) \neq 0. 
	\]	
\end{lem}
\begin{proof}
	For $\lambda \neq \lambda'$, the result follows by the definition of $\chi_\lambda^{\downarrow}$.	
So, now onwards we assume $\lambda = \lambda'$. From  \cite[Theorem~5.12.5 ]{MR3287258} the character values of  $\chi_\lambda^{\pm}(w)$ are given as below:
	\[
	\chi_\lambda^{\pm}(w) =  \begin{cases}  \frac{1}{2}\chi_\lambda(w), &\,\, \text{if} \,\,   \mathcal{F}(\text{type}(w)) \neq \lambda, \\  
	\frac{1}{2}(\epsilon_{\text{type}(w)} \pm \sqrt{\frac{\epsilon_{\text{type}(w)} n!}{|C_{\text{type}(w)}| }}), & \,\, \text{if} \,\,   \mathcal{F}(\text{type}(w)) = \lambda \,\, \text{and}\,\, w \in C_{\text{type}(w)}^+ .\\
	\end{cases}
	\]
	
where, for $\mu = (2m_1+1, 2m_2+1, ..)$ with all distinct parts, 
$\epsilon_\mu = (-1)^{\sum m_i}.$ We also have 
	\[
	\chi_\lambda^\pm(vwv^{-1}) = \chi_\lambda^\mp (w) 
	\]
	for any odd permutation $v$ (see \cite[(4.16)]{MR3287258} ).
	Combining above results we obtain:
	\[
	\chi_\lambda(w) =  \begin{cases} 2 \chi_\lambda^{\pm}(w), &\,\, \text{if} \,\,   \mathcal{F}(\text{type}(w)) \neq \lambda, \\  
	\epsilon_{\text{type}(w)}, & \,\, \text{otherwise}.
	\end{cases}
	\]
	The lemma follows from the above expression along with the observation that $\chi_\lambda(w)$ and $\chi_\lambda^{\pm}(w)$ are both non-zero for  $\mathcal{F}(\text{type}(w)) = \lambda$. 
	\end{proof}

\subsection{Double Cover of the Symmetric and Alternating groups}
\label{double cover}
Schur studied the double cover of the symmetric group and its representations. His primary interest in this question was motivated by the study of projective representations of the symmetric groups. Schur proved that for $S_n$, there exists a group $\widetilde{S}_n$ such that the following is a central extension
\[
1 \rightarrow \mathbb Z/2\mathbb Z \rightarrow \widetilde{S}_n \xrightarrow{\pi} S_n
\]
and every projective representation of $S_n$ lifts to an ordinary representation of $\widetilde{S}_n$. Any group with the above properties is called either Schur cover or representation group of $S_n$. Schur classified these groups up to isomorphism and proved that up to isomorphism, two representations groups of $S_n$ for $n \geq 4$. It is well known that there exists a dimension preserving bijection between the irreducible representations of both of these groups. We refer the reader to Morris~\cite{MR136668} and Stembridge~\cite{MR991411} for details regarding the representation groups of $S_n$ and their irreducible representations. Throughout this section, we fix the following double cover of $S_n$. 
\[
\widetilde{S}_n = \langle  z, t_1, t_2, \ldots, t_{n-1}, \mid z^2 = 1,  t_j^2 = z, \,\, (t_j t_{k})^2 = z \,\, \text{for}\,\, |j-k| \geq 2, \,\, (t_j t_{j+1})^3 = z   \rangle.
\]  
 
Let $\widetilde{A}_n$ be the preimage of $A_n$ in $\widetilde{S}_n$. The group $\widetilde{A}_n$ is a double cover of $A_n$ for all $n$. This is uniquely determined upto isomorphism. For $n \neq 6,7$, this group is further known to be the representation group of $A_n$.

For $G = S_n$ or $A_n$, let $ \widetilde G = \widetilde{S}_n$ or $\widetilde{A}_n$ respectively. The group $ \widetilde G$ projects onto $G$ via $pi$, therefore any irreducible representation of $G$ naturally gives an irreducible of $ \widetilde G$, we call these as $G$-representation of $ \widetilde G$. Any representation of $ \widetilde G$ that is not obtained from $G$ via projection map  $\pi$ is called a spin representation of $ \widetilde G$. 
\begin{remark}
	\label{remark:non-spin-characters}
Let $\rho$ be a representation of $G$ and $\widetilde{\rho}$ be the corresponding $G$-representation of $\ \widetilde G$, then
\[
\chi_{\widetilde{\rho}}(g) = \chi_\rho(\pi(g)),\,\,\text{for all}\,\, g\in  \widetilde G.
\]
This implies that $\chi_{\widetilde{\rho}}(g) \neq 0 $ if and only if $\chi_{\rho}(\pi(g)) \neq 0 $. We note that for any prime $p$, $g $ is $p$-regular implies $\pi(g)$ is $p$-regular. Conversely, if $h \in G$ is $p$-regular then there exists $g' \in  \widetilde G$ such that $g'$ is $p$-regular and $\pi(g') = h $ (see also Remark~\ref{remark:2-regular-preimage}). Therefore $\widetilde{\rho}$  is quasi $p$-Steinberg if and only if $\rho$ is quasi $p$-Steinberg. Therefore for $G$-representations of $ \widetilde G$, we obtain quasi $p$-Steinberg representations of $ \widetilde G$ from Theorem~\ref{thm:classification-symmetric}.  
For this reason, we shall focus only on the spin characters of $ \widetilde G$.  
\end{remark}

Now we proceed to describe the irreducible spin representations of $\widetilde{S}_n$ and $\widetilde{A}_n$. All result mentioned in this section are from Hoffman--Humphreys~\cite{MR1205350} and Stembridge~\cite{MR991411}. Also see Bessenrodt~\cite{MR3856528} and references therein. 
Let $D(n)$ (respectively, $O(n)$) denote the set of partitions of $n$ with all parts distinct (respectively, odd). Let $D^{+}(n)$ (respectively, $D^{-}(n)$) consists of partitions in $D(n)$ with an even (respectively, odd) number of even parts.  Any $\lambda \in O(n) \cap D(n)$ satisfies $\lambda \in D^+(n)$. The partition $\lambda = \emptyset $ is assumed to be in $D^{+}(n)$.

 Given a partition $\alpha \vdash n$, let $\widetilde{C}_\alpha$ be the pre-image under $\pi$ of the $S_n$-conjugacy class $C_{\alpha}$. The following theorem gives the conjugacy classes of $\widetilde{S}_n$ due to Schur. 
 
\begin{theorem}
Given $\alpha \in \Lambda(n)$, the set  $\widetilde{C}_\alpha$ splits into two $\widetilde{S}_n$ conjugacy classes if and only if $\alpha\in O(n)\cap D^{-}(n)$, otherwise $\widetilde{C}_\alpha$ does not split.
\end{theorem}

Let $\sigma_\alpha \in \widetilde{S}_n$ denotes the preimage of $w_\alpha \in S_n$ obtained by fixing the preimage of $(i, i+1)$ to be $t_i$ for all $1 \leq i \leq n-1$. The other possible preimage of $w_\alpha \in S_n$ is $z \sigma_\alpha$. We note  that $\chi(z \sigma_\alpha)$ is easily obtained from  $\chi(\sigma_\alpha)$ for an irreducible character $\chi$ because $z$ is a central element of order two. 
\begin{remark} 
	\label{remark:2-regular-preimage} 
We note that if $w_\alpha$ is $2$-regular then either $\sigma_\alpha$ or $z \sigma_\alpha$ has order equal to that of $w_\alpha$ and hence is $2$-regular. We will call such a preimage to be a $2$-regular preimage of $w_\alpha$. 
\end{remark} 

  The irreducible spin characters of $\widetilde{S}_n$ are indexed by $\lambda \in D(n)$. For $\lambda \in D^{+}(n)$, there is a self associate spin character $\widetilde{\chi}_\lambda $ and for $\lambda \in D^{-}(n)$ there are pair of associate characters $\widetilde{\chi}_\lambda^{+}$ and $\wchi_\lambda^{-}$.  
  The next theorem, due to Schur, helps to compute character values.

\begin{theorem}
	\label{thm:symmetric-spin-characters1} For $\lambda \in D(n)$, the character values for spin character $\widetilde{\chi}_\lambda$ of $\widetilde{S}_n$ satisfy the following:

	\begin{enumerate} 
		\item $\widetilde{\chi}_\lambda(\sigma_\alpha) = 0$ for all $\alpha \notin D^{-}(n) \cup O(n)$, $\lambda \in D(n)$.
		\item $\widetilde{\chi}_\lambda(\sigma_\alpha) = 0$ for all $\alpha \in D^{+}(n)$ such that $\alpha \neq \lambda$, $\lambda \in D^{+}(n)$.  
		\item $\wchi_\lambda^{ +}(\sigma_\alpha) = \wchi_\lambda^{-}(\sigma_\alpha)$ for all $\alpha \in O(n)$, $\lambda \in D^{-}(n)$. 
		\item  For $\lambda = (\lambda_1, \lambda_2,\cdots, \lambda_{\ell(\lambda)} )\in D^{-}(n)$, we have 
		\[
		\wchi_\lambda^{+} (\sigma_\lambda) = -\wchi_\lambda^{-} (\sigma_\lambda) = i^{(n-\ell(\lambda)+1)/2} \sqrt{\prod_j \lambda_j/2}
		.\]
	\end{enumerate} 
\end{theorem}
\begin{remark}
In general, we will write an irreducible spin character of $\widetilde{S}_n$ corresponding to $\lambda$ by $\widetilde{\chi}_\lambda$ and mention $\widetilde{\chi}_\lambda^\pm$ only if explicitly required. 
\end{remark}

For $\alpha \in O(n)$ and $\lambda \in D(n)$, the character value is determined by Morris recursion formula. Morris ~\cite{MR136668} developed the character theory of spin representations of the symmetric group by introducing the shifted diagram, bar length analogous to the Young diagram, hook length, respectively. The shifted diagram corresponding to a partition $\lambda$ with all distinct parts is obtained by placing $\lambda_i$ boxes starting at the diagonal position $(i, i)$. The shift symmetric diagram $SS(\lambda)$ of $\lambda$ with $2n$ boxes is obtained by flipping the shifted diagram of $\lambda$ over the diagonal and then gluing it to the shifted diagram.
 
 The bar length $b_{ij}$ of a box $(i,j)$ is the hook length of the box  $(i,j)$ in the shift symmetric diagram. The dimension of the irreducible spin representation of $S_n$ corresponding the partition $\lambda$ is given by the \emph{bar length formula }: 
 \begin{displaymath}
 \widetilde{\chi}_\lambda(1)= 2^{[n-l(\lambda)/2]} \frac{n!}{B_{\lambda}},
 \end{displaymath}
where $B_{\lambda}$ is the product of all bar lengths of $\lambda$.
\begin{example}
For the partition $\lambda=(5, 2,1)$ we depict the associated shifted diagram and also put the bar lengths in $SS(\lambda)$.
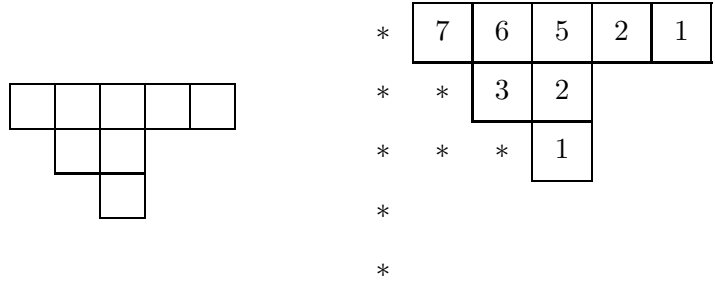
\begin{displaymath}
\ydiagram{5,1+2,2+1}
\quad 
\quad \quad \quad 
\ytableausetup
{mathmode, boxsize=2em}
\begin{ytableau}
\none[*] & 7 & 6 & 5 & 2 & 1 \\
\none[*] & \none[*] & 3 & 2  & \none \\
\none[*] & \none[*] &\none[*] & 1 & \none \\
\none[*] & \none & \none & \none  & \none \\  
\none[*] & \none & \none & \none  & \none \\  
\end{ytableau}
\end{displaymath}
Here $B_{\lambda}=7\cdot 6 \cdot 5 \cdot 2 \cdot 3 \cdot 2 $. 
\end{example}
To state the main result we need to define the notion of bar removal. The removal of a bar of length $l$ from  a partition $\lambda= (\lambda_1, \lambda_2, \ldots, \lambda_k) \in D(n)$  corresponds to removing a part $\lambda_i$ of length $l$ (if exists) or removing two parts $\lambda_i$ and $\lambda_j$ of lengths $r$ and $l-r$ respectively (if exist) or subtracting $l$ from part $\lambda_i$, if the resultant partition belongs to $D(n-l)$. An $l$-bar corresponds a hook of length $l$ in $SS(\lambda)$ and therefore an $l$-bar removal from $\lambda \in D(n)$ corresponds to removing hook of length $l$ from $SS(\lambda)$. The leg length $L(b)$ of a $l$-bar $b$ is the leg length of the associated $l$- hook in shift symmetric diagram.
\begin{example}
We can remove a $2$-bar from $\lambda=(5,2,1)$ in two different ways: removing the second part from $\lambda$ or subtracting $2$ from the first part of $\lambda$. Hence the resultant partitions are $(5,1)$ and $(3,2,1)$ respectively. On the other hand, there is only one way to remove $3$-bar from $\lambda$: removing the second and third parts from $\lambda$.
\end{example}

The number of ways of an $l$-bar removal from $\lambda \in D(n)$ corresponds to the number of hooks of length $l$ in $SS(\lambda)$. Now onwards,  we will always use the language of shift symmetric diagrams for the explicit computations. As mentioned earlier, the following result is useful in determining the character values for spin representations. 

\begin{theorem}(Morris recursion formula)
	\label{thm:M-N-double}
	Let $\lambda \in D(n)$ and let $\alpha \in O(n)$ with $l$ as a part. Then,
	\[
	\wchi_\lambda(\sigma_\alpha) = \sum_{b \text{ is a } l- bar} (-1)^{L(b)} 2^{m(b)} \widetilde{\chi}_{\lambda \setminus b}(\sigma_{\alpha \setminus l}), 
	\]
	where
\[
m(b) = \begin{cases}  1,\,\, \text{ if} \,\, \epsilon(\lambda \setminus b) - \epsilon(\lambda) = 1, \\ 0,  \,\, \text{otherwise}.  \end{cases}
\]	
and
\[
\epsilon(\lambda)  = \begin{cases} 0, \,\, \lambda \in D^{+}(n),  \\ 1, \,\, \lambda \in D^{-}(n). \end{cases}
\]

\end{theorem}
For any partition $\alpha$, the number of its nonzero parts is called the length of $\alpha$ and is denoted by $\ell(\alpha)$.  

\begin{lem}
	\label{lem:spin-n-character} Let $\lambda = (n)$ and $\widetilde{\chi}_\lambda$ be a spin character of $\widetilde{S}_n$. The  values of the character $\widetilde{\chi}_\lambda$ is zero on all conjugacy classes except the following: 
	\begin{enumerate}  
		\label{lem: (n)-spin-vlaues}
		\item For odd $n$,
		\[
		\widetilde{\chi}_\lambda(\sigma_\alpha) = 2^{\frac{\ell(\alpha) -1}{2}}, \,\,\text{for}\,\, \alpha \in O(n).
		\]
		\item For $n = 2k$,
		\[
		\widetilde{\chi}_\lambda^{\pm}(\sigma_\alpha) = \begin{cases}
		2^{\frac{\ell(\alpha)-2}{2}},& \text{for}\,\, \alpha \in O(n), \\
		\pm i^k\sqrt{k}, & \text{for} \,\,\alpha = (n).
		\end{cases}
		\]
	\end{enumerate}	
\end{lem}
\begin{theorem} 
	\label{thm:restriction-from-double-cover-Sn-to-An}
	\begin{enumerate}
		\item 	For each $\lambda \in D^{-}(n)$, the restriction of $\widetilde{\chi}_\lambda^{\pm}$ to $\widetilde{A}_n$ gives one irreducible spin character 
		\[
		\wchi_\lambda^{ +}|_{\widetilde{A}_n} = 	\wchi_\lambda^{ -}|_{\widetilde{A}_n} = \langle\wchi_\lambda\rangle^\downarrow.
		\]
		\item For each $\lambda \in D^{+}(n)$, the restriction of $\widetilde{\chi}_\lambda^{\pm}$ to $\widetilde{A}_n$ gives two conjugate irreducible spin characters
		\[
		\wchi_\lambda|_{\widetilde{A}_n} = \langle\widetilde{\chi}_{\lambda}\rangle^+ + \langle\widetilde{\chi}_{\lambda}\rangle^-.
		\]
		\item $\langle\wchi_\lambda\rangle^{\pm}$ vanishes on classes projecting to a cycle type $\mu \neq \lambda$ that is not in $O(n)$. 
		\item For $\lambda = (\lambda_1, \lambda_2,\cdots, \lambda_{\ell(\lambda)} ) \in D^{+}(n)$, define 
		\[
		\Delta^{\lambda}(\sigma) = \begin{cases}  \pm i^{\frac{n-\ell(\lambda)}{2}} \sqrt{\prod_j \lambda_j},  & \text{if}\,\, \text{type}(\sigma) = \lambda, \\ 
		0, & \text{otherwise}.
		\end{cases}
		\]
		Then we have 
		\[
		\langle\wchi_\lambda\rangle ^{\pm}(\sigma) = \frac{1}{2}(\wchi_\lambda \pm \Delta^{\lambda})(\sigma).
		\]
	\end{enumerate}
\end{theorem}
We shall write $\langle\wchi_\lambda\rangle ^{*}$ to denote $\langle\wchi_\lambda\rangle^\downarrow$ for $\lambda\in D^{-}(n)$ and $\langle\wchi_\lambda\rangle ^{\pm}$ for $\lambda\in D^{+}(n)$. The following result follows directly from Theorem~\ref{thm:restriction-from-double-cover-Sn-to-An}(4). 
\begin{cor}
	\label{cor:double-cover-alternating-vs-symmetric}
	Let $\lambda \in D(n)$ and $\tau \in \widetilde{A}_n$ such that $\text{type}(\tau) \neq \lambda$. Then $\widetilde{\chi}_\lambda(\tau) \neq 0$ if and only if $\langle\wchi_\lambda\rangle^*(\tau) \neq 0$. 
	
\end{cor}

	\section{Proof of Theorems~\ref{thm:classification-symmetric} -- \ref{thm:weak-Steinberg-double-cover} }
	\label{section:proofs}
\subsection{Symmetric and Alternating groups}
	\begin{prop}
		\label{prop:odd p-alternating}
		Let $p$ be an odd prime and $\lambda \vdash n$ such that $\lambda \neq (n), (1^n)$. Let $\chi_\lambda^{*}$ be a quasi $p$-Steinberg character of $A_n$.  Then the possible $\lambda$'s are the following or their conjugates. 

\begin{enumerate}[label=(\alph*)]
\item For $n = 3$: $\lambda = (2,1)$ for  $p = 3$. 
\item For $n=4$:
\begin{itemize}
	\item $\lambda = (3,1)$ for $p =3$.
	\item  $\lambda = (2,2)$ for $p= 3$. 
\end{itemize} 
\item For $n = 5$:
\begin{itemize}
	\item $\lambda = (3, 1^2)$ for $p = 3$.
	\item $\lambda = (3,2)$ for $p=5$.
\end{itemize}
\item	For $ n =6$:
\begin{itemize} 
\item $\lambda = (4,2)$ for $p = 3$.
\item $\lambda = (5,1)$ for $p=5$.
\item $\lambda = (3,3)$ for $p=5$. 
\end{itemize} 
\item For $n = 9$: $\lambda = (7,2)$ for $p= 3$.  

\end{enumerate} 
\end{prop}
	
\begin{proof} Let $\lambda \vdash n$ such that $\chi_\lambda^*$ is a quasi $p$-Steinberg character of $A_n$.  For the proof, we consider various  cases depending on $p \mid n$ or $p \nmid n$ separately along with the parity of $n$. 
	
	\vspace{.2cm}
\noindent {\bf Case 1a}: Let $p \nmid n$ and $n$ is odd. Let $\alpha = (n)$ then $w_\alpha \in A_n$. By the hypothesis, $\chi_{\lambda}^{*}(w_\alpha) \neq 0$. 
	This along with Lemma~\ref{lemma:deduction-ofsymmetric-nonzero} implies $\chi_\lambda(w_\alpha) \neq 0$. By \mn{1}, we have $wt_{n}(\lambda) \neq 0$. Therefore $\lambda = (n-m, 1^m)$ for some $m \geq 0$. We can further assume, without loss of generality, that $m \leq n-m-1$. Therefore, we have $H(\lambda) = n\cdot(n-m-1)! \cdot m!$ (see Figure~\ref{Hook diagram 1}).
	\begin{figure}[h]
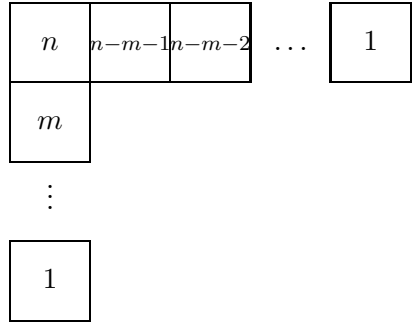

	\ytableausetup
{mathmode, boxsize=2.7 em}
	\begin{ytableau} 
		n & \scriptstyle \tiny{n-m-1} & \scriptstyle \tiny{n-m-2} & \none[\dots] & 1
		\cr
	 m
		\cr \none[\vdots]
	\cr 1	
		\end{ytableau}
		\caption{Hook diagram for $\lambda = (n-m, 1^m) $}\label{Hook diagram 1}
	\end{figure}

We first assume that $m \geq 3$. For $\alpha_1 = (n-2, 1,1)$ and $\alpha_2 = (n-3,2,1) $, $w_{\alpha_1}, w_{\alpha_2} \in A_n$.  We  observe (see Figure~\ref{Hook diagram 1}) that $wt_{n-2}(\lambda) = 0$ and $wt_{n-3}(\lambda) = 0$. Therefore by \mn{1} and Lemma~\ref{lemma:deduction-ofsymmetric-nonzero}, we have $\chi_{\lambda}^*(w_{\alpha_i}) = 0$ for $i=1, 2$. Further, $p$ is odd; therefore at least one of the $w_{\alpha_i}$ is $p$-regular. This is a contradiction to the quasi $p$-Steinberg condition. Hence, we must have $m \leq 2$.  
	
We next assume that $m = 2$ and $p \nmid n-2$. Consider $\alpha = (n-2, 1, 1)$. Then by the hypothesis and Lemma~\ref{lemma:deduction-ofsymmetric-nonzero}, we have $\chi_\lambda(w_\alpha) \neq 0$. Therefore by \mn{1}, $\lambda$ has a hook of length $n-2$. By considering the hook lengths and the fact that $n-2$ is odd, we obtain $n -2 = 1$. Therefore $ n =3$. Since $A_3$ is abelian so only one dimensional representations are obtained in this case and all of these are quasi $3$-Steinberg. This is included in (a) of Proposition.

 We now consider $m = 2$ and $p \mid n-2$. In this case, we have the following:
\begin{enumerate}[label = (\roman*)]
	\item For $ p \neq 3$, we have $n > 5$. By \mn{2},  $\chi_\lambda^*(\alpha) = 0$ for $\alpha = (n-4, 3, 1)$. Since $w_\alpha$ is $p$-regular, $\chi_\lambda^*$ is not quasi $p$-Steinberg.

	\item For $p = 3$ and $n \geq 7$, we use  $ \alpha = (n-6, 5, 1)$  as above to obtain that $\chi_\lambda^*$ is not quasi $3$-Steinberg. For $p=3$ and $n = 5$, we have $\lambda = (3, 1^2)$, therefore $\lambda = \lambda'$ in this case. By the character values, it follows that this is a quasi $3$-Steinberg character and is included in (c) of Proposition. 
\end{enumerate} 

At last, we consider the remaining case of $m =1$. For $n=3$, we have already seen argument above. So now onwards, we assume that $n \geq 5$. In this case, either $p \nmid n-3$ or $p\nmid n-5$. In the first case, we use $\alpha  = (n-3, 2, 1)$ and in the second case we use  $\alpha = (n-5, 4, 1)$ combined with \mn{2} as above to obtain a contradiction to quasi $p$-Steinberg condition.

\noindent {\it \bf Case 1b}: Let $p \nmid n$ and $n$ is even. We deal  the cases $p$ divides $n-1$ or not separately.
	
\noindent \textbf{Case 1b.1} For this case, we assume that $p\nmid n-1$. Let $\alpha = (n-1, 1)$, so $w_\alpha$ is a $p$-regular element of $A_n$. Then by hypothesis,  $\chi_\lambda^*(w_\alpha) \neq 0$. This by Lemma~\ref{lemma:deduction-ofsymmetric-nonzero} implies $\chi_\lambda(w_\alpha) \neq 0$. By $\mn{1}$, $\lambda$  must have a hook of length $n-1$. Since $\lambda \neq (n)$ so, $\lambda=(n-m-2,2,1^m)$ for some  $m \geq 0$. 
\begin{figure}[h]
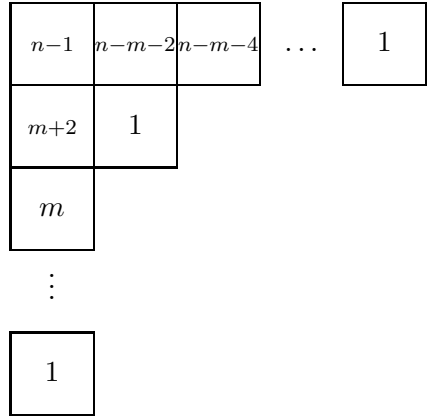

\ytableausetup
{mathmode, boxsize=2.8 em}
\begin{ytableau} 
\scriptstyle n-1 & \scriptstyle n-m-2 & \scriptstyle n-m-4 & \none[\dots] & 1 
\cr
\scriptstyle m+ 2 & 1 
\cr
m \cr 
\none[\vdots]
\cr
1
\end{ytableau}
\caption{Hook diagram for $\lambda=(n-m-2,2,1^m)$}\label{Hook diagram 2}
\end{figure}
In this case (see Figure~\ref{Hook diagram 2}),
\[
H(\lambda) = (n-1) (n-m-2) (m+2) (m!)(n-m-4)!.
\] 

Without loss of generality, we can assume that $2m +4\geq n $.  We consider the case of $m$ even and odd separately. 

{\bf If $m$ is odd:}  In this case, $n \geq 6$.   Consider $\alpha_1 = (m+4, 1^{n-m-4})$ and $\alpha_2 = (m+3, 2, 1^{n-m-5})$. Because $m+2 \geq n-m-2$ and both $n-1$ and $m$ are odd, therefore $wt_{m+3}(\lambda) = 0$. Similarly, $wt_{m+4}(\lambda) = 0$ unless $n-1 = m+4$. We note that atleast one of the $w_{\alpha_i}$ is $p$-regular. By $\mn{1}$, if $w_{\alpha_1}$ is $p$-regular then $\chi_\lambda^*$ is not quasi $p$-Steinberg. Otherwise $w_{\alpha_2}$ is $p$-regular and we must have $m+4 = n-1$ in this case. Therefore if $\chi_\lambda^*$ if quasi $p$-regular then $n = m+5$, $\lambda = (3, 2, 1^m)$. We next note that if $p \mid m+ 3$, then $p \nmid m+1$ and we consider  $\alpha = (m+1, 4)$ and observe that $\chi_\lambda^*(\alpha) = 0$. This gives a contradiction to $\chi_\lambda^*$ being quasi $p$-Steinberg for any prime $p$ such that $p \mid m+ 3$. If $p \nmid m+3$, then we consider  $ \alpha = (m+3, 2)$ and again obtain the contradiction to quasi $p$-Steinberg condition as before. 
 
{\bf If $m$ is even:} In this case $n \geq 4$. For $n = 4$, we must have $\lambda = (2,2)$. For this case, $\chi_\lambda^{*}$ has dimension one so these are quasi $3$-Steinberg and is included in (b) of Proposition. 

Now onwards we assume that $n \geq 6$.  We consider $\alpha_1 = (m+3, 1^{n-m-3})$ and $\alpha_2 = (m+4, 2, 1^{n-m-6})$. Then $wt_{m+4}(\lambda) = 0$ because $m+4 (> m+2)$ is even and $n-1$ is odd. Similarly,  $wt_{m+3}(\lambda) = 0$ unless $n-1 = m+3$. By the argument as for $m$ odd, we obtain that if $\chi_\lambda^*$ is quasi $p$-Steinberg then $n-1 = m+3$. So $n = m + 4$ and $\lambda = (2, 2, 1^m)$. We consider the hook lengths for this $\lambda$. 

If $ p \nmid m+1$, then $m+1$ is odd and so by arguments as above, we must have $m+1 = 1$, i.e. $m = 0$. This contradicts the assumption $n \geq 6$. Therefore, we do not obtain any quasi $p$-Steinberg $\chi_\lambda^*$ in this case.

 On the other hand, if $ p \mid  m+1$ and $n (= m+4 )> 6$ then we consider $\alpha = (m, 4)$.  Again by \mn{2}, we obtain that $\chi_\lambda^*$ is not quasi $p$-Steinberg. For $n = 6$, the assumption  $p \mid n-3$ gives $p = 3$. However for $p = 3$ and $n = 6$, we will obtain $p \mid n$. This is a contradiction to our assumption that $p \nmid n$. So this case also does not arise.

\vspace{.2cm}
\noindent 	\textbf{Case 1b.2}: In this case, $n$ is even, $p \nmid n$ and $ p \mid n-1$. In particular, $p\nmid n-2$. Let $\alpha = (n-2, 2)$. Then $w_\alpha \in A_n$ and by hypothesis and Lemma~\ref{lemma:deduction-ofsymmetric-nonzero} we must have $\chi_\lambda^*(w_\alpha) \neq 0$. Then by \mn{1},  $wt_{n-2}(\lambda) \neq 0$. This implies, $\lambda=(n-1, 1)$, $(n-2, 2)$, $(n-m-3,3,1^{m})$ or their conjugate. 

We first assume that $\lambda = (n-1, 1)$. Let $\alpha = (n-6, 3, 2, 1)$, then $w_\alpha \in A_n$. By $\mn{2}$, we obtain $\chi_\lambda^*(\alpha) = 0$.

We observe that for $p \neq 3, 5$ and $n \geq 8$, $w_\alpha$ is $p$-regular so we obtain that $\chi_\lambda^*$ is not quasi $p$-Steinberg in this case. Therefore $n < 7$ for $p>5$. But this case never arises as $p \mid n-1$. This implies, we must have $ p = 3$ or $ p = 5$. For $p = 3$ and $n \geq 9$,  the partition $\alpha = (n-8, 5, 2, 1)$ will give the required contradiction. Similarly, $p = 5$, $n \geq 9$, $\alpha = (n-8, 4, 3, 1)$ will lead to a contradiction. So the only remaining cases are 
 $\lambda = (3,1)$ for $p =3$ and $\lambda = (5, 1)$ for $p=5$. These are included in (b) and (d) of the Proposition.

Next, we consider $\lambda = (n-2, 2)$, see Figure~\ref{Hook diagram 2} for hook lengths.  Now observe that  $n-3$ does not appear as a hook length in $H(\lambda)$. Since $p \nmid n-3$ and $n-3$ is odd, we must have $n-3 = 1$. Therefore $\lambda = (2, 2)$ and is included in (b) of Proposition.   

At last, we consider $\lambda = (n-m-3,3,1^{m})$. Also consider the hook diagram below: 

\begin{figure}[h]
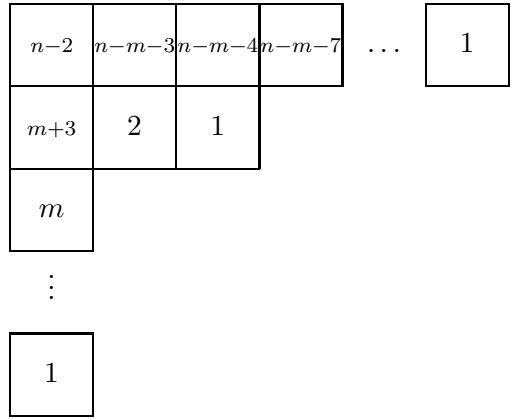

\ytableausetup
{mathmode, boxsize=2.8 em}
\begin{ytableau} 
\scriptstyle n-2 & \scriptstyle n-m-3 & \scriptstyle n-m-4 & \scriptstyle n-m-7 & \none [\dots] & 1 
 \cr
 \scriptstyle m+3 & 2 & 1
 \cr
 m
 \cr
 \none[\vdots]
 \cr
 1
\end{ytableau}
\caption{Hook diagram for $\lambda=(n-m-3,3,1^m)$}\label{Hook diagram 3}
\end{figure}
Since $p \nmid n-3$, $\lambda$ must have a hook of length $n-3$. Then 
$n-3=n-m-3$ or $n-3=m+3$ (as indicated above).
This implies either $\lambda = (n-3, 3)$ or $(3, 3, 1^m)$. Consider their hook diagrams:

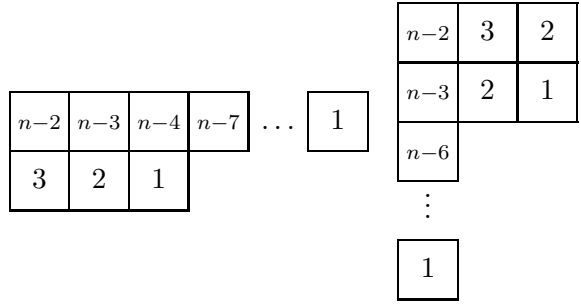
\[
\ytableausetup
{mathmode, boxsize=2 em}
\begin{ytableau} 
\scriptstyle n-2 & \scriptstyle n-3 & \scriptstyle n-4 & \scriptstyle n-7 & \none [\dots] & 1 
 \cr
3 & 2 & 1
\end{ytableau}
\quad 
\ytableausetup
{mathmode, boxsize=2 em}
\begin{ytableau} 
 \scriptstyle n-2 & 3 & 2
 \cr
\scriptstyle n-3 & 2 & 1
 \cr
 \scriptstyle n-6
 \cr
 \none[\vdots]
 \cr
 1
\end{ytableau}
\]

In both cases, since $p\nmid n-5$, by the hypothesis $\lambda$ must have a hook of length $n-5$. Then either $n-5 = 1$ or $n-5 =3$. For $n = 6$, we obtain $\lambda = (3, 3)$, which is quasi $5$-Steinberg and is included in (d) of Proposition. For $n = 8$, we obtain a contradiction by using  $\alpha = (4, 4)$.

		\vspace{.2cm}
\noindent 	\textbf{Case 2a}:  In this case, we consider $p\mid n$ and $n$ is odd. Then $p \nmid n-2$ and $n-2$ is odd. So by $\mn{1}$, $\lambda$ must have hook length equal to $n-2$. Therefore we obtain that $\lambda$ must be either $(n-1, 1)$ or $(n-2, 2)$ or  $(n-m-3,3,1^{m})$ or their conjugate. 

For $\lambda = (n-1, 1)$. The hook lengths of this case are indicated in Figure~\ref{Hook diagram 1}. For $p=3$ and $n \geq 7$, we consider $\alpha = (n-5, 4, 1)$. By using $\mn{2}$, we observe that $\chi_\lambda^*(w_\alpha) = 0$ and $w_\alpha$ is $p$-regular. This implies that $\chi_\lambda^*$ is not quasi $p$-Steinberg. For $n = 6$ and $p =3$, $\chi_{(5,1)}^*(w_{(5,1)}) = 0$ and therefore is not quasi $p$-Steinberg representation. The only remaining case is of $p = 3$, $n = 4$ and $\lambda = (3,1)$ and this is included in (b) of Proposition. For $p \neq 3$, we consider $\alpha=(n-3, 2, 1)$ and again see that $\chi_\lambda^*$ is not quasi $p$-Steinberg representation. 

For $\lambda = (n-2, 2)$, see Figure~\ref{Hook diagram 2} for hook lengths.   For $p = 3$ and $n-8\geq 3$,  $ \alpha = (n-8, 4, 4)$ combined with \mn{2} gives the contradiction to $\chi_\lambda^*$ being quasi $p$-Steinberg. Therefore we must have  $n=9$.  We check by direct computations that $n = 9$ with $\lambda = (7, 2)$ is quasi $3$-Steinberg and is included in (e) of Proposition. For $p \neq 3$,  
we observe that $p \nmid n -3$ and $n-3$ is even. By $\mn{2}$, $\chi_\lambda^*(\alpha) = 0$  for $ \alpha = (n-3, 2, 1)$ unless $n-3 =2$ . Therefore, the only possible case is $n = 5$, $p =5$ and $\lambda = (3,2)$, which indeed  is a quasi 5-Steinberg character of $A_5$. This is included in (c) of Proposition.

Next, we consider $\lambda = (n-m-3,3,1^{m})$. See Figure~\ref{Hook diagram 3} for hook lengths in this case. For $p = 3$, we have $p \nmid n-4$ and $n-4$ is odd so $\lambda$ must have a hook length equals $n-4$. We obtain that $\lambda$ has following possible shapes: 
\[
(n-4, 3, 1), (n-3, 3), (4, 3, 1^{n-7}), (3, 3, 1^{n-6}).
\]
We deal these cases as below and obtain contradiction using (\textbf{MN1}), (\textbf{MN2}).
\begin{itemize} 
	\item For $\lambda =  (n-4, 3, 1)$ and $n \geq 11$, we use $\alpha = (n-8,5,1,1,1)$ to obtain required contradiction.  The only remaining case is when $n=9$. For this case, we use $\alpha = (4,2,1,1,1)$ to obtain a contradiction. 
	\item For $\lambda = (n-3, 3)$, $\alpha = (n-5, 2, 1,1,1)$ gives the required contradiction. 
	\item For $\lambda = (4, 3, 1^{n-7})$, we use $\alpha = (n-5, 4, 1)$ to obtain a contradiction.
	\item For $\lambda = (3, 3, 1^{n-6} )$, we use $\alpha = (n-4, 1,1,1)$. This along with $p = 3$ and $p \mid n$ implies $\chi_\lambda^*$ is not quasi $p$-Steinberg. 
	\end{itemize}

If $p \neq 3$ then $p \nmid n-3$ and from now onwards we are in the same case as that of Case~1b.2 with $\lambda = (n-m-3,3,1^{m})$ (see figure~\ref{Hook diagram 3}) and conclude that $\chi_\lambda^*$ is not quasi $p$-Steinberg representation.

		\vspace{.2cm}
\noindent	\textbf{Case 2b}: Let $p\mid n$ and $n$ is even. Then $p \nmid n-1$  and $n-1$ is odd. Therefore, we must have a hook of length $n-1$.  This implies $\lambda = (n-m-2,2,1^m)$ for some $m \geq 0$ in this case. The argument in this  case follows parallel to Case~1b.1 except the fact that $p=3$ and $n=6$ is possible. This gives $\lambda = (2, 2, 1^2)$. We check by direct computations that $\chi_\lambda^*$ for $\lambda = (2, 2, 1^2)$ is quasi $3$-Steinberg and is included in (d) of Proposition.  
\end{proof}

In the next Proposition, we consider the case of $p = 2$ and prove the statement parallel to above for this case. 
\begin{prop}
	\label{prop: even p-alternating}
Let $p = 2$. The only possible $\lambda$ for which the $A_n$ character $\chi_\lambda^*$ for $* \in \{ \pm , \downarrow \}$ is quasi 2-Steinberg  are the following or their conjugates: 
\begin{itemize}
	\item[(a)]  $\lambda =  (2,1)$.
		\item[(b)] $ \lambda = (2, 2)$. 
	\item[(c)] $\lambda = (4,1)$.
	\item[(d)] $\lambda = (3, 2, 1)$.
	\item[(e)] $\lambda = (5,2,1)$.
\end{itemize}

\end{prop}
\begin{proof} 
We first assume that $n$ is odd. Then for $\alpha = (n)$, the element $w_\alpha \in A_n$ is $2$-regular.  Therefore by the hypothesis,  Lemma~\ref{lemma:deduction-ofsymmetric-nonzero} and \mn{1}, $\lambda$ itself must be a hook of the form  $(n-m,1^m)$. Without loss of generality, we can further assume that $n - m- 1 \geq m$. For this $\lambda $, hook lengths are given in figure~\ref{Hook diagram 1}. Note that if $n-(n-m-1)=m+1>2$, then there exists an odd number $q$ such that $n-m-1 < q < n$ and there is no hook of length $q$. This will contradict the hypothesis because for $\alpha = (q, 1^{n-q})$, element $w_\alpha$ is $2$-regular and $\chi_\lambda^*(w_\alpha) = 0$. Therefore we must have $m\leq 1$, i.e. $\lambda=(n-1,1)$. For $n \geq 7$, the character value $\chi_{(6,1)}^*$ is zero on $w_\alpha$ for $ \alpha = (n-4, 3, 1)$ by $\mn{2}$ and $w_\alpha$ is a $2$-regular element of $A_n$. This implies that either $n = 3$ or $n = 5$. For $n=3 $ and $n = 5$, it follows by direct computations that $\chi_{(2,1)}^*$ and $\chi_{(4,1)}^*$ are quasi $2$-Steinberg representations.

\vspace{.2cm}
\noindent \textbf{Case2}: We next assume that $n$ is even. Then $p$ does not divide odd number $n-1$. So $\lambda$ has must have hook length equals $n-1$. This in particular implies, $\lambda=(n-m-2, 2, 1^m)$ for $m \geq 0$. The hook lengths for this $\lambda$ are given in figure~\ref{Hook diagram 2}. Without loss of generality, we can further assume that $n-m-4 \geq m$. If $(n-1)-(n-m-2) > 2$ then there exists an odd number $q$ lying between $n-m-2$ and $n-1$ which does not appear as a hook number. This by \mn{1} gives contradiction to being quasi 2-Steinberg. So we must have $(n-1)-(n-m-2) \leq 2$. But this implies that $m \leq 1$. For $m = 0$ and $n \geq 6$, the odd hook length $n-3$ is missing for $\lambda = (n-2, 2)$. Therefore this case is not possible. This implies that for $m = 0$, we must have $\lambda = (2, 2)$ which is one dimensional so is quasi $2$-Steinberg. 
For $m = 1$ and $n -5 > 3$, we consider $\alpha (n-5,5)$. Then $w_\alpha$ is $2$-regular and $\chi_\lambda^*(w_\alpha) = 0$ by \mn{2}. Therefore for $m = 1$, the only possible cases are either $n = 6$ or  $n=8$.  By direct computations, we obtain that both $\chi_{(3,2,1)}^*$ and $\chi_{(5, 2, 1)}^*$ are quasi $2$-Steinberg characters of $A_6$ and $A_8$ respectively. 

\end{proof}

\begin{proof}[Proof of Theorems~\ref{thm:classification-symmetric} and \ref{thm:classification-alternating}]
	The proof of Theorem~\ref{thm:classification-alternating}
	follows by Propositions~\ref{prop:odd p-alternating}, \ref{prop: even p-alternating} along with the fact that for $\lambda \neq \lambda'$, we have $\chi_\lambda^\downarrow = \chi_{\lambda'}^\downarrow$.  
	
	We proceed to complete the argument for Theorem~\ref{thm:classification-symmetric}. 
	By the character values given in Lemma~\ref{lemma:deduction-ofsymmetric-nonzero} and the fact that the set of $p$-regular elements of $A_n$ is a subset of $p$-regular elements of $S_n$, it follows that the $\chi_\lambda$ is a quasi $p$-Steinberg representation of $S_n$ implies $\chi_\lambda^*$ for $* \in \{ \pm, \downarrow  \}$ is a quasi $p$-Steinberg representation of $A_n$. To complete the proof, we consider  all the $\lambda$'s obtained in Theorem~\ref{thm:classification-alternating}. 
	Out of these, by using direct computations, we collect all $\ lambda$s such that $\chi_\lambda$ is quasi $p$-Steinberg characters of $S_n$ and obtain  Theorem~\ref{thm:classification-symmetric}.

\end{proof}

\subsection{Double cover of Symmetric and Alternating groups}

\begin{prop} 
	\label{prop:double-cover-odd-prime}
	For odd prime $p$,  both $\widetilde{S}_n$ and $\widetilde{A}_n$ do not have any quasi $p$-Steinberg spin character. 
\end{prop}
\begin{proof} The irreducible spin characters of  $\widetilde{S}_n$ are parametrized by $\lambda \in D(n)$, i.e. the partitions of $n$ with all distinct parts (see Section \ref{double cover}). For $\lambda \in D(n)$, let $\widetilde{\chi}_\lambda$ (respectively, $\widetilde{\chi}_\lambda^*$) be a quasi $p$-Steinberg character of $\widetilde{S}_n$ (respectively, $\widetilde{A}_n$). We consider the case of $n$ even and $n$ odd separately.
	
\textbf{Case 1}:{ \bf  $n$ is even.}
For $n=4$,  we must have $\lambda = (3,1)$. In this case, we consider $\alpha = (2, 2)$ and note that $\sigma_\alpha$ is $p$-regular. Since $\alpha \notin O(n) \cup D^-(n)$, by Theorem~\ref{thm:symmetric-spin-characters1}(1) and Theorem~\ref{thm:restriction-from-double-cover-Sn-to-An}(3), we have $\widetilde{\chi}_\lambda(\sigma_\alpha) = 0$ and $\langle\widetilde{\chi}_\lambda\rangle^*(\sigma_\alpha) =0$. Therefore, this $\lambda$ does not give a quasi $p$-Steinberg representation of $\widetilde{S}_n$ and  $\widetilde{A}_n$.

Now onwards, we assume that $n \geq 6$.  Consider $\alpha_1 = (n-2, 2)$ and $\alpha_2 = (n-4, 4)$. Then at least one of the $\sigma_{\alpha_i}$ for $1 \leq i \leq 2$ is $p$-regular. Since $\alpha_1, \alpha_2 \notin O(n) \cup D^{-}(n)$, by Theorem~\ref{thm:symmetric-spin-characters1}(1), we have $  \widetilde{\chi}_\lambda (\sigma_{\alpha_i}) = 0$ for $i =1,2$. This implies $\widetilde{\chi}_\lambda$ is not quasi $p$-Steinberg character of $\widetilde{S}_n$ for $\lambda \in D(n)$. 

It remains to deal $\langle\widetilde{\chi}_\lambda\rangle^*$ for $\widetilde{A}_n$. By Theorem~\ref{thm:restriction-from-double-cover-Sn-to-An}(3), and considering $\alpha_1$ and $\alpha_2$ as above, we again obtain that $\langle\widetilde{\chi}_\lambda\rangle^\pm$ is not quasi $p$-Steinberg for $\widetilde{A}_n$ unless either $\lambda = (n-2, 2)$ or $\lambda = (n-4, 4)$. Now we handle these two remaining cases. First we assume that $\lambda = (n-2, 2)$. Considering again $\alpha = (n-4, 4)$ and hypothesis, we can further assume that $p \mid n-4$. Therefore, we must have $n \geq 10$. For $\beta = (n-8, 8)$,  the element $\sigma_\beta$ is $p$-regular. By Theorems~ \ref{thm:symmetric-spin-characters1} and \ref{thm:restriction-from-double-cover-Sn-to-An}(3), we have $\langle\widetilde{\chi}_\lambda\rangle^*(\sigma_\beta) = 0$. This gives a contradiction to the quasi $p$-Steinberg property for $\lambda = (n-2, 2)$. At last, we assume that $\lambda = (n-4, 4)$. As above, we can assume that $p \mid n-2$. Therefore, we can further assume that $n \geq 8$. In this case we consider $\alpha = (n-4, 2)$. Then $\sigma_\alpha$ is $p$-regular and by Theorem~\ref{thm:symmetric-spin-characters1}, we have $\widetilde{\chi}_\lambda(\sigma_\alpha) = 0$. Since $\text{type}(\alpha) \neq \lambda$, so by Theorem~\ref{thm:restriction-from-double-cover-Sn-to-An}(3) we obtain $\langle\widetilde{\chi}_\lambda\rangle^\pm(\sigma_\alpha) = 0$. This again gives a contradiction to the quasi $p$-Steinberg condition. 
	
\vspace{.3cm}	
\textbf{Case 2:} {\bf $n$ is odd}.
For $n=5$, either $\lambda = (4,1)$ or $\lambda = (3,2)$. For both of these cases, we consider $\alpha = (2, 2,1)$ and we obtain contradiction in this case parallel to $n=4$ case. 

Now onwards, we assume that $n \geq 7$.  Consider $\alpha_1 = (n-3, 2,1)$ and $\alpha_2 = (n-5, 4, 1)$. Then at least one of the $\sigma_{\alpha_i}$ for i = 1, 2, is $p$-regular. Since $\alpha_1, \alpha_2 \notin O(n) \cup D^{-}(n)$, by Theorem~\ref{thm:symmetric-spin-characters1}(1) the character  $\widetilde{\chi}_\lambda$ of $\widetilde{S}_n$ is zero on both $\sigma_{\alpha_i}$. This gives a contradiction to the quasi $p$-Steinberg property for $\widetilde{S}_n$. 

For $\widetilde{A}_n$, the argument is parallel to the even $n$ case with $n \geq 6$. In this case, the only cases to consider are $\lambda = (n-3, 2,1)$ or $\lambda = (n-5, 4, 1)$. For $\lambda = (n-3, 2, 1)$ we consider $\alpha_1 = (n-4, 2, 2)$ and $\alpha_2 = (n-5, 2,1)$. For $\lambda = (n-5, 4, 1)$, we consider $\alpha_1 = (n-3, 2,1)$ and $\alpha_2 = (n-4, 2, 2)$. It is easy to obtain contradiction using these for these $\lambda$'s.

\end{proof} 
\begin{proof}[Proof of Theorem~\ref{thm:classification-double-cover}] 
Let $\lambda \in D(n)$, such that $\widetilde{\chi}_\lambda$ (respectively, $\langle\widetilde{\chi}_\lambda\rangle^*$) be a quasi $p$-Steinberg character of $\widetilde{S}_n$ (respectively, $\widetilde{A}_n$). 

For $p \neq 2$, the result follows by Proposition~\ref{prop:double-cover-odd-prime}.  Now onwards we assume that $p = 2$.  For $\lambda = (n)$, by Lemma~\ref{lem:spin-n-character} and Theorem~\ref{thm:restriction-from-double-cover-Sn-to-An}, we obtain that  $\widetilde{\chi}_\lambda$ and $\langle\widetilde{\chi}_\lambda\rangle^*$ is quasi $2$-Steinberg character for $\widetilde{S}_n$ and $\widetilde{A}_n$ respectively.
For rest of the proof, we assume that $\lambda \neq (n)$. The proof below works uniformly for both $\widetilde{S}_n$ as well as $\widetilde{A}_n$. An important role in this proof is played by the fact that if $g \in \widetilde{A}_n$ is $p$-regular such that $\text{type}(g) \neq \lambda$ then $\widetilde{\chi}_\lambda(g) \neq 0$. For $\widetilde{S}_n$, this result directly follows from hypothesis and for $\widetilde{A}_n$,  it is obtained from  Corollary~\ref{cor:double-cover-alternating-vs-symmetric} combined with the hypothesis. We shall use this argument in the proof without specifically mentioning it. Now onwards, for every $2$-regular element $w_\alpha \in S_n$, we fix $\sigma'_\alpha \in \widetilde{S}_n$ such that $\sigma'_\alpha$ is $2$-regular preimage of $w_\alpha$(for existence see Remark~\ref{remark:2-regular-preimage}). Recall that for such a choice of $\sigma'_\alpha$, we have  $\widetilde{\chi}_\lambda(\sigma'_\alpha) = 0$ if and only if $\widetilde{\chi}_\lambda(\sigma_\alpha) = 0$.

\textbf{Case 1:} $p \nmid n$. Let $\sigma'_{(n)} \in \widetilde{S}_n$ be a $2$-regular lift of $w_{(n)}$. By hypothesis and the fact that $\sigma'_{(n)}$ is $2$-regular, we have  $\widetilde{\chi}_{\lambda}(\sigma_{(n)})\neq 0$. This combined with Theorem~\ref{thm:M-N-double} implies there must be a hook of length $n$ in $SS(\lambda)$. So $\lambda$ must be $(n-m,m)$ with $m<n-m$ and $m\geq 0$. The hook numbers for $SS(\lambda)$ are depicted in Figure~\ref{Hook diagram 4}. 
\begin{figure}[h]
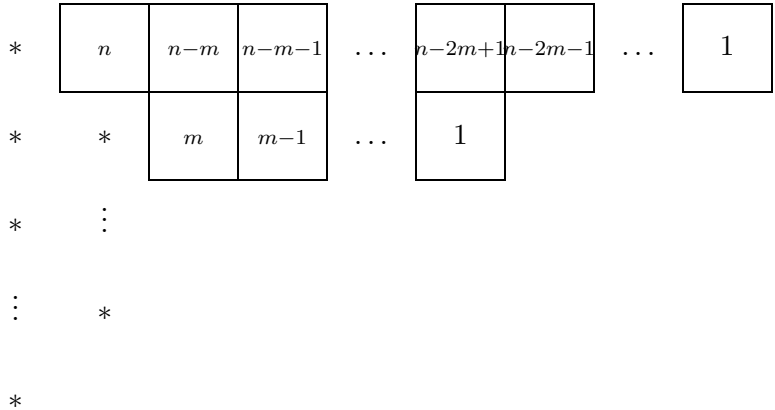

	\begin{displaymath}
	\ytableausetup
	{mathmode, boxsize=3 em}
	\begin{ytableau} 
\none[*] & \scriptstyle n & \scriptstyle n-m & \scriptstyle n-m-1 &  \none [\dots] & \scriptstyle n-2m + 1 & \scriptstyle n-2m-1 & \none[\dots] & 1
	\cr
	\none[*]  & \none[*] & \scriptstyle m & \scriptstyle m-1 & \none [\dots] & 1 
	\cr
	\none[*] & 	\none [\vdots]
	\cr
	\none [\vdots] & \none[*]
		\cr
\none[*]
	\end{ytableau}  
	\end{displaymath}
	\caption{Shift symmetric hook diagram for $\lambda=(n-m,m)$}\label{Hook diagram 4}
\end{figure}

Again  $\widetilde{\chi}_{\lambda}(\sigma_{(n-2, 1^2)})\neq 0$ as $\sigma'_{(n-2, 1^2)}$ is 2-regular. Hence there must be a hook of length $n-2$ in $SS(\lambda)$.	This implies $n-2=n-m-i$ as $n-2 = m$  not possible for  $n \geq 3$. Then $m = 1,2$. For $m=2$, there is no hook of length $n-4$ for $n \geq 7$. But $\sigma_{(n-4)}$ is 2-regular which contradicts the quasi-Steinberg property.  Therefore for $m = 2$ we must have $n \leq 6$.  The only possibility is $\lambda = (3, 2)$. By the direct computations, we observe that $\lambda = (3,2)$ gives a quasi 2-Steinberg character for $\widetilde{S}_n$ as well as for $\widetilde{A}_n$. For $m=1$, there is no hook of length $n-2$ in $SS((n-1,1))$ unless $n=3$. This does not contribute to any spin character.

\textbf{Case 2:}	$p \mid n$. Since $p\nmid n-1$, as above $SS(\lambda)$ must contain a hook of length $n-1$. Therefore, $\lambda$ must be of the form either $(n)$ or $(n-1,1)$ or $(n-m-1,m,1)$. 

 For $\lambda = (n-1, 1)$ with $n-5 > 1$, the character value $\widetilde{\chi}_\lambda$ is zero on  $\sigma'_{(n-5, 3, 1^2)}$. This gives a contradiction to the quasi p-Steinberg character. So  $n\leq 6$ and we must have $\lambda = (3,1)$ or $\lambda = (5,1)$. Both of these are quasi $2$-Steinberg characters for $\widetilde{S}_n$ and $\widetilde{A}_n$. 
For $\lambda = (n-m-1,m,1)$, the hook numbers for $SS(\lambda)$ are depicted in Figure~\ref{Hook diagram 5}.
 \begin{figure} [h]
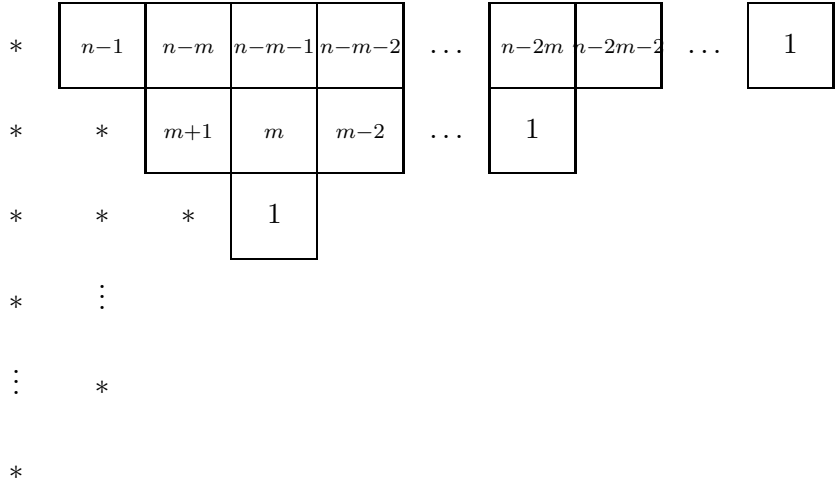

	\begin{displaymath}
	\ytableausetup
	{mathmode, boxsize=2.9 em}
	\begin{ytableau} 
\none[*] & \scriptstyle n-1 & \scriptstyle n-m & \scriptstyle n-m-1 & \scriptstyle n-m-2 &  \none [\dots] & \scriptstyle n-2m  & \scriptstyle n-2m-2 & \none[\dots] & 1
	\cr
	\none[*]  & \none[*] & \scriptstyle m+1 & \scriptstyle m & \scriptstyle m-2 &\none [\dots] & 1 
	\cr
	\none[*] & 	\none[*] & \none[*] & 1
	\cr
	\none[*] & \none [\vdots]
	
		\cr
	\none [\vdots] & \none[*]
		\cr
\none[*]
	\end{ytableau}  
	\end{displaymath}
	\caption{Shift symmetric hook diagram for $\lambda=(n-m-1,m,1)$}\label{Hook diagram 5}
\end{figure}
Since $\sigma'_{(n-3,1^3)}$ is 2-regular,  $\widetilde{\chi}_\lambda(\sigma_{(n-3,1^3)})\neq 0$. Note that $n>2m+1$ as $\lambda$ has distinct parts. For $m \geq 2$ we must have $n-3=n-m$ or $n-3 = n-m-1$. That is $m=3$ or $m= 2$. For $m = 3$,  we consider $\alpha = (n-5, 1^5)$. We must have $n-5 = m=3$. So we get $\lambda=(4,3,1)$. This case is not possible because its character values is zero on an element of type $(3)(3)$. For $m=2$, we again consider  $\alpha = (n-5, 1^5)$.  From hook lengths of $SS(\lambda)$, we obtain that either $n-5 = 3$ or $n-5 = 1$. So $\lambda = (3,2,1)$ for $n=6$ or $\lambda = (5,2,1)$ for $n=8$. Both of these give quasi $2$-Steinberg characters for $\widetilde{S}_n$ and $\widetilde{A}_n$.  
	
\end{proof} 	

\subsection{Proof of Corollaries~\ref{thm:weak-Steinberg-symmetric-alternating} and \ref{thm:weak-Steinberg-double-cover} }We first note that every weak $p$-Steinberg character is a quasi $p$-Steinberg. Conversely, by \cite[Theorem~1]{MR4042}, a quasi $p$-Steinberg character $\chi$ of $G$ is a weak $p$-Steinberg character if and only if $\chi(1) = |G|_p$. From Theorems~\ref{thm:classification-symmetric} and \ref{thm:classification-alternating}, we obtain the complete list of quasi $p$-Steinberg characters of $S_n$ and $A_n$. From the dimension computations, it is easy to verify that the weak $p$-Steinberg characters of $S_n$ and $A_n$ are exactly the ones mentioned in Corollary~\ref{thm:weak-Steinberg-symmetric-alternating}. 

	For $\widetilde{S}_n$ and $\widetilde{A}_n$, we note that for $\lambda = (3,1), (3,2), (3,2,1), (5,1), (5,2,1)$,  representations do not satisfy the required dimension hypothesis. For $\lambda = (n)$, we observe that the dimension of the corresponding representation is $2^{\lfloor (n-1)/2 \rfloor }$.  To complete the proof in this case, we claim that $2^{\lfloor (n-1)/2 \rfloor } < 2 (2^{\nu_2(n!)})$. Therefore enough to prove that $\lfloor (n-1)/2 \rfloor \leq  \nu_2(n!) $. For this we use the Legendre's formula, i.e. for any prime $p$ the valuation of $p$ in $n!$ is given by the following:
	\[
\nu_p(n!) = \sum_{i=1}^\infty \lfloor n/p^i \rfloor.
	\] 
	For our case, $p =2$,  we have $\lfloor \frac{n}{2}\rfloor = \lfloor \frac{n-1}{2}\rfloor$  for odd $n$ and $\lfloor \frac{n}{2}\rfloor > \lfloor \frac{n-1}{2}\rfloor$ for even $n$. Therefore, $\lfloor (n-1)/2 \rfloor \leq  \nu_2(n!)$  for $n\geq 4$. For $n<4$, the result follows by direct computation. This completes the proof of Corollary~\ref{thm:weak-Steinberg-double-cover}.

\section{Further Discussion}
\label{section:further-questions}
Given a positive integer $n$, a partition $\lambda \vdash n$ is called of maximal $p$-weight if $wt_{p}(\lambda)=\lfloor \frac{n}{p} \rfloor$.  Given a prime $p$, let 
\begin{displaymath}
A(n,p):=\lbrace \lambda \vdash n \mid wt_{p'}(\lambda)=\lfloor \frac{n}{p'} \rfloor, \text{for all prime } p'\neq p\rbrace.
\end{displaymath}
It can be shown that if $\chi_{\lambda}$, $\lambda \vdash n$, is a quasi $p$-Steinberg character then $\lambda\in A(n,p)$.  This motivates the following question.
\begin{que}
 Determine $A(n,p)$ for any  prime $p$. 
\end{que}
For $p=2$, $A(n,p)$ is determined in \cite[Theorem 1.1]{MR2120100}. In addition to the above, the following questions seem interesting. 
\begin{que}
What can be the possible dimensions of a $p$-Steinberg character of a finite group?
\end{que}
From Theorems~\ref{thm:classification-symmetric}, \ref{thm:classification-alternating}, and \ref{thm:classification-double-cover}, we obtain that any quasi $p$-Steinberg character  of groups $S_n$, $A_n$ and their double covers have dimension $p^t$ for some $t \in \mathbb Z_{\geq 0}$. The same answer is not true in general. For example, the group $G=\mathrm{PSL}(2,11)$ has two irreducible quasi 5-Steinberg characters of dimension 10.
\begin{que} Determine the quasi $p$-Steinberg characters of all finite simple groups.
\end{que}

\section*{Acknowledgements}
\label{section:acknowledgements}
This research was initiated during the International Centre for Theoretical Sciences (ICTS) program-Group Algebras, Representations and Computation(Code:  ICTS/Prog-garc2019/10) and supported in part by the ICTS. The authors are grateful to Prof. Dipendra Prasad for his questions about Steinberg characters and discussions after that. They warmly thank Prof. Bessenrodt for her comments on this work and provide the above example of $\mathrm{PSL}(2,11)$. The second named author acknowledges the support of the UGC CAS-II grant (Grant No. F.510/25/CAS-II/2018(SAP-I)) and SERB MATRICS grant (MTR/2018/000094).  This work would not have been possible without the extensive computations using Sagemath and  GAP.
\bibliography{reference}{}

\begin{thebibliography}{10}

\bibitem{MR1853455}
{\sc A.~Balog, C.~Bessenrodt, J.~r.~B. Olsson, and K.~Ono}, {\em Prime power
  degree representations of the symmetric and alternating groups}, J. London
  Math. Soc. (2), 64 (2001), pp.~344--356.

\bibitem{MR2335704}
{\sc C.~Bessenrodt}, {\em Bar weights of bar partitions and spin character
  zeros}, J. Algebraic Combin., 26 (2007), pp.~107--124.

\bibitem{MR3856528}
\leavevmode\vrule height 2pt depth -1.6pt width 23pt, {\em Critical classes,
  {K}ronecker products of spin characters, and the {S}axl conjecture}, Algebr.
  Comb., 1 (2018), pp.~353--369.

\bibitem{MR1920404}
{\sc C.~Bessenrodt and J.~r.~B. Olsson}, {\em Prime power degree
  representations of the double covers of the symmetric and alternating
  groups}, J. London Math. Soc. (2), 66 (2002), pp.~313--324.

\bibitem{MR2120100}
\leavevmode\vrule height 2pt depth -1.6pt width 23pt, {\em Weights of
  partitions and character zeros}, Electron. J. Combin., 11 (2004/06),
  pp.~Research Paper 5, 13.

\bibitem{MR4042}
{\sc R.~Brauer and C.~Nesbitt}, {\em On the modular characters of groups}, Ann.
  of Math. (2), 42 (1941), pp.~556--590.

\bibitem{MR794307}
{\sc R.~W. Carter}, {\em Finite groups of {L}ie type}, Pure and Applied
  Mathematics (New York), John Wiley \& Sons, Inc., New York, 1985.
\newblock Conjugacy classes and complex characters, A Wiley-Interscience
  Publication.

\bibitem{MR201524}
{\sc C.~W. Curtis}, {\em The {S}teinberg character of a finite group with a
  {$(B,\,N)$}-pair}, J. Algebra, 4 (1966), pp.~433--441.

\bibitem{MR1382032}
{\sc M.~R. Darafsheh}, {\em {$p$}-{S}teinberg characters of alternating and
  projective special linear groups}, J. Algebra, 181 (1996), pp.~196--206.

\bibitem{MR1247494}
{\sc W.~Feit}, {\em Extending {S}teinberg characters}, 153 (1993), pp.~1--9.

\bibitem{MR1205350}
{\sc P.~N. Hoffman and J.~F. Humphreys}, {\em Projective representations of the
  symmetric groups}, Oxford Mathematical Monographs, The Clarendon Press,
  Oxford University Press, New York, 1992.
\newblock $Q$-functions and shifted tableaux, Oxford Science Publications.

\bibitem{MR876960}
{\sc J.~E. Humphreys}, {\em The {S}teinberg representation}, Bull. Amer. Math.
  Soc. (N.S.), 16 (1987), pp.~247--263.

\bibitem{MR4048469}
{\sc G.~Malle and A.~Zalesski}, {\em Steinberg-like characters for finite
  simple groups}, J. Group Theory, 23 (2020), pp.~25--78.

\bibitem{MR136668}
{\sc A.~O. Morris}, {\em The spin representation of the symmetric group}, Proc.
  London Math. Soc. (3), 12 (1962), pp.~55--76.

\bibitem{MR1264418}
{\sc J.~r.~B. Olsson}, {\em Combinatorics and representations of finite
  groups}, vol.~20 of Vorlesungen aus dem Fachbereich Mathematik der
  Universit\"{a}t GH Essen [Lecture Notes in Mathematics at the University of
  Essen], Universit\"{a}t Essen, Fachbereich Mathematik, Essen, 1993.

\bibitem{MR3513936}
{\sc M.~A. Pellegrini and A.~E. Zalesski}, {\em On characters of {C}hevalley
  groups vanishing at the non-semisimple elements}, Internat. J. Algebra
  Comput., 26 (2016), pp.~789--841.

\bibitem{MR3287258}
{\sc A.~Prasad}, {\em Representation theory}, vol.~147 of Cambridge Studies in
  Advanced Mathematics, Cambridge University Press, Delhi, 2015.
\newblock A combinatorial viewpoint.

\bibitem{MR1676282}
{\sc R.~P. Stanley}, {\em Enumerative combinatorics. {V}ol. 2}, vol.~62 of
  Cambridge Studies in Advanced Mathematics, Cambridge University Press,
  Cambridge, 1999.
\newblock With a foreword by Gian-Carlo Rota and appendix 1 by Sergey Fomin.

\bibitem{MR80669}
{\sc R.~Steinberg}, {\em Prime power representations of finite linear groups},
  Canadian J. Math., 8 (1956), pp.~580--591.

\bibitem{MR87659}
\leavevmode\vrule height 2pt depth -1.6pt width 23pt, {\em Prime power
  representations of finite linear groups. {II}}, Canadian J. Math., 9 (1957),
  pp.~347--351.

\bibitem{MR991411}
{\sc J.~R. Stembridge}, {\em Shifted tableaux and the projective
  representations of symmetric groups}, Adv. Math., 74 (1989), pp.~87--134.

\bibitem{MR1425574}
{\sc P.~H. Tiep}, {\em {$p$}-{S}teinberg characters of finite simple groups},
  J. Algebra, 187 (1997), pp.~304--319.

\end{thebibliography}
\bibliographystyle{siam}
\end{document}